\documentclass[a4paper,11pt]{amsart}
\pdfoutput=1
\usepackage{amssymb,amsmath,amsthm,amsfonts,centernot}
\usepackage[colorlinks=true,linkcolor=blue,citecolor=blue,urlcolor=cyan, pdfpagelabels=false]{hyperref}
\usepackage{amsmath,tikz}
\usetikzlibrary{matrix}
\usepackage{dsfont}
\usepackage{perpage} 
\MakePerPage{footnote}
\begin{document}
\numberwithin{equation}{section}
\newtheorem{theorem}{Theorem}
\newtheorem{algo}{Algorithm}
\newtheorem{lem}{Lemma} 
\newtheorem{de}{Definition} 
\newtheorem{ex}{Example}
\newtheorem{pr}{Proposition} 
\newtheorem{claim}{Claim} 
\newtheorem*{re}{Remark}
\newtheorem*{asi}{Aside}
\newtheorem{co}{Corollary}
\newtheorem{conv}{Convention}
\newcommand{\di}{\hspace{1.5pt} \big|\hspace{1.5pt}}
\newcommand{\idi}{\hspace{.5pt}|\hspace{.5pt}}
\newcommand{\hs}{\hspace{1.3pt}}
\newcommand{\thmf}{Theorem~1.15$'$}
\newcommand{\ndi}{\centernot{\big|}}
\newcommand{\nidi}{\hspace{.5pt}\centernot{|}\hspace{.5pt}}
\newcommand{\lpp}{\mbox{$\hspace{0.12em}\shortmid\hspace{-0.62em}\alpha$}}
\newcommand{\btt}{\mbox{$\raisebox{-0.59ex}
  {${{l}}$}\hspace{-0.215em}\beta\hspace{-0.88em}\raisebox{-0.98ex}{\scalebox{2}
  {$\color{white}.$}}\hspace{-0.416em}\raisebox{+0.88ex}
  {$\color{white}.$}\hspace{0.46em}$}{}}
  \newcommand{\un}{\hs\underline{\hspace{5pt}}\hs}
\newcommand{\lp}{\widehat{\lpp}}
\newcommand{\bt}{\hspace*{2pt}\widehat{\hspace*{-2pt}\btt}}
\newcommand{\PQ}{\bb{P}^1(\bb{Q})}
\newcommand{\pmn}{\cl{P}_{M,N}}
\newcommand{\he}{holomorphic eta quotient\hspace*{2.5pt}}
\newcommand{\hes}{holomorphic eta quotients\hspace*{2.5pt}}
\newcommand{\defG}{Let $G\subset\GG$ be a subgroup that is conjugate to a finite index subgroup of $\G$. } 
\newcommand{\defg}{Let $G\subset\GG$ be a subgroup that is conjugate to a finite index subgroup of $\G$\hs\hs} 
\renewcommand{\phi}{\varphi}
\newcommand{\Z}{\bb{Z}}
\newcommand{\ZD}{\Z^{\D}}
\newcommand{\N}{\bb{N}}
\newcommand{\Q}{\bb{Q}}
\newcommand{\A}{\widehat{A}}
\newcommand{\pii}{{{\pi}}}
\newcommand{\R}{\bb{R}}
\newcommand{\C}{\bb{C}}
\newcommand{\I}{\hs\cl{I}_{n,N}}
\newcommand{\St}{\operatorname{Stab}}
\newcommand{\D}{\cl{D}_N}
\newcommand{\rh}{{{\boldsymbol\rho}}}
\newcommand{\bh}{{\cl{M}}} 
\newcommand{\lv}{\hyperlink{level}{{\text{level}}}\hspace*{2.5pt}}
\newcommand{\fct}{\hyperlink{factor}{{\text{factor}}}\hspace*{2.5pt}}
\newcommand{\q}{\hyperlink{q}{{\mathbin{q}}}}
\newcommand{\rd}{\hyperlink{redu}{{{\text{reducible}}}}\hspace*{2.5pt}}
\newcommand{\ird}{\hyperlink{irredu}{{{\text{irreducible}}}}\hspace*{2.5pt}}
\newcommand{\str}{\hyperlink{strong}{{{\text{strongly reducible}}}}\hspace*{2.5pt}}
\newcommand{\rdn}{\hyperlink{redon}{{{\text{reducible on}}}}\hspace*{2.5pt}}
\newcommand{\atl}{\hyperlink{atinv}{{\text{Atkin-Lehner involution}}}\hspace*{3.5pt}}
\newcommand{\T}{\mathrm{T}}
\newcommand{\nm}{{N,M}}
\newcommand{\mn}{{M,N}}
\renewcommand{\H}{\fr{H}}
\newcommand{\W}{\text{\calligra W}_n}
\newcommand{\GG}{\cl{G}}
\newcommand{\g}{\fr{g}}
\newcommand{\Gm}{\Gamma}
\newcommand{\Gmtl}{\widetilde{\Gamma}_\ell}
\newcommand{\gm}{\gamma}
\newcommand{\go}{\gamma_1}
\newcommand{\gmt}{\widetilde{\gamma}}
\newcommand{\gmdt}{\widetilde{\gamma}'}
\newcommand{\gmot}{\widetilde{\gamma}_1}
\newcommand{\gmdot}{{\widetilde{\gamma}}'_1}
\newcommand{\s}{\Large\text{{\calligra r}}\hspace{1.5pt}}
\newcommand{\ms}{m_{{{S}}}}
\newcommand{\nisim}{\centernot{\sim}}
\newcommand{\level}{\hyperlink{level}{{\text{level}}}}
\newcommand{\Redcon}{the \hyperlink{red}{\text{Reducibility~Conjecture}}}
\newcommand{\Conred}{Conjecture~$1'$}
\newcommand{\Conredd}{Conjecture~$1''$}
\newcommand{\Conreddd}{Conjecture~$1'''$}
\newcommand{\Conired}{Conjecture~$2'$}
\newtheorem*{pro}{\textnormal{\textit{Proof of the proposition}}}
\newtheorem*{cau}{Caution}
\newtheorem{thrmm}{Theorem}[section]
\newtheorem{no}{Notation}
\renewcommand{\thefootnote}{\fnsymbol{footnote}}
\newtheorem{oq}{Open question}
\newtheorem{conj}{Conjecture}
\newtheorem{hy}{Hypothesis} 
\newtheorem{expl}{Example}
\newcommand\ileg[2]{\bigl(\frac{#1}{#2}\bigr)}
\newcommand\leg[2]{\Bigl(\frac{#1}{#2}\Bigr)}
\newcommand{\e}{\eta}
\newcommand{\sgn}{\operatorname{sgn}}
\newcommand{\bb}{\mathbb}
\newtheorem*{conred}{Conjecture~\ref{con1}$\mathbf{'}$}
\newtheorem*{conredd}{Conjecture~\ref{con1}$\mathbf{''}$}
\newtheorem*{conreddd}{Conjecture~\ref{con1}$\mathbf{'''}$}
\newtheorem*{conired}{Conjecture~\ref{19.1Aug}$\mathbf{'}$}
\newtheorem*{procl}{\textnormal{\textit{Proof}}}
\newtheorem*{thmff}{Theorem~\ref{7.4Jul}$\mathbf{'}$}
\newtheorem*{coo}{Corollary~\ref{17Aug}$\mathbf{'}$}
\newcommand{\cooo}{Corollary~\ref{17Aug}$'$}
\newtheorem*{cotw}{Corollary~\ref{17.1Aug}$\mathbf{'}$}
\newcommand{\cotww}{Corollary~\ref{17.1Aug}$'$}
\newtheorem*{cothr}{Corollary~\ref{17.2Aug}$\mathbf{'}$}
\newcommand{\cothre}{Corollary~\ref{17.2Aug}$'$}
\newtheorem*{cne}{Corollary~\ref{15.5Aug}$\mathbf{'}$}
\newcommand{\cnew}{Corollary~\ref{15.5Aug}$'$\hspace{3.5pt}}
\newcommand{\fr}{\mathfrak}
\newcommand{\cl}{\mathcal}
\newcommand{\rad}{\mathrm{rad}}
\newcommand{\ord}{\operatorname{ord}}
\newcommand{\m}{\setminus}
\newcommand{\G}{\Gamma_1}
\newcommand{\GN}{\Gamma_0(N)}
\newcommand{\X}{\widetilde{X}}
\renewcommand{\P}{{\textup{p}}} 
\newcommand{\al}{{\hs\operatorname{al}}}
\newcommand{\p}{p_\text{\tiny (\textit{N})}}
\newcommand{\pN}{p_\text{\tiny\textit{N}}}
\newcommand{\U}{u_\textit{\tiny N}}
\newcommand{\Upr}{u_{\textit{\tiny N}^\prime}}
\newcommand{\Up}{u_{\textit{\tiny p}^\textit{\tiny e}}}
\newcommand{\Un}{u_{\textit{\tiny p}_\textit{\tiny 1}^{\textit{\tiny e}_\textit{\tiny 1}}}}
\newcommand{\Um}{u_{\textit{\tiny p}_\textit{\tiny m}^{\textit{\tiny e}_\textit{\tiny m}}}}
\newcommand{\Ut}{u_{\text{\tiny 2}^\textit{\tiny a}}}
\newcommand{\At}{A_{\text{\tiny 2}^\textit{\tiny a}}}
\newcommand{\Uh}{u_{\text{\tiny 3}^\textit{\tiny b}}}
\newcommand{\Ah}{A_{\text{\tiny 3}^\textit{\tiny b}}}
\newcommand{\Uprl}{u_{\textit{\tiny N}_1}}
\newcommand{\Uprlm}{u_{\textit{\tiny N}_i}}
\newcommand{\UM}{u_\textit{\tiny M}}
\newcommand{\UMp}{u_{\textit{\tiny M}_1}}
\newcommand{\w}{\omega_\textit{\tiny N}}
\newcommand{\wm}{\omega_\textit{\tiny M}}
\newcommand{\wa}{\omega_{\text{\tiny N}_\textit{\tiny a}}}
\newcommand{\wma}{\omega_{\text{\tiny M}_\textit{\tiny a}}}
\renewcommand{\P}{{\textup{p}}}

\title[\tiny{Finiteness of simple holomorphic eta quotients of a given weight}]
{Finiteness of simple holomorphic eta quotients of a given weight}

\author{Soumya Bhattacharya}
\address 
{CIRM : FBK\\
via Sommarive 14\\
I-38123 Trento}

\email{soumya.bhattacharya@gmail.com}

\maketitle

 \begin{abstract}
  We 
  provide a simplified proof of Zagier's conjecture / Mersmann's theorem
  which states 
  that\hspace{.85pt} of any particular weight, 
  there are only finitely many holomorphic eta quotients,
  none of which is an integral rescaling of another eta quotient
  or a product of two 
  holomorphic eta quotients other than 1 and itself.
 \end{abstract}
 \section{Introduction}
 The Dedekind eta function is defined by the infinite product:\hypertarget{queue}{}
 \begin{equation} \eta(z):=q^{\frac{1}{24}}\prod_{n=1}^\infty(1-q^n)\hspace{5pt}\text{for all $z\in\H$,}
\label{17.4Aug}\end{equation} 
where $q^r=q^r(z):=e^{2\pi irz}$ for all $r$ 
and 
$\H:=\{\tau\in\C\hs\idi\operatorname{Im}(\tau)>0\}$. 
Eta is a holomorphic function on $\H$ with no zeros.
This function comes up naturally in many areas of Mathematics (see the Introduction in \cite{B-three} for a brief
overview of them). 
The function $\e$ is a modular form
of weight $1/2$ with a multiplier system 
on $\operatorname{SL}_2(\Z)$ (see \cite{b}).
An 
eta quotient $f$ is a finite product of the form 
\begin{equation}
 \prod\e_d^{X_d},
\label{13.04.2015}\end{equation}
where $d\in\N$, $\eta_d$ is the rescaling of $\eta$ by $d$, defined by
\begin{equation}
 \e_d(z):=\e(dz) \ \text{ for all $z\in\H$}
\end{equation}
and $X_d\in\Z$.
Eta quotients naturally inherit modularity 
from $\e$: The eta quotient $f$ in (\ref{13.04.2015}) transforms like a modular form of
weight $\frac12\sum_dX_d$ with a multiplier system on suitable congruence subgroups of $\operatorname{SL}_2(\Z)$: The largest among
these subgroups is 
\begin{equation}
 \Gm_0(N):=\Big{\{}\begin{pmatrix}a&b\\ c&d\end{pmatrix}\in
\operatorname{SL}_2(\Z)\hspace{3pt}\Big{|}\hspace{3pt} c\equiv0\hspace*{-.3cm}\pmod N\Big{\}},
\end{equation}
where 
\begin{equation}
 N:=\operatorname{lcm}\{d\in\N\hs\idi\hs X_d\neq0\}.
\end{equation}
We call $N$ the \emph{level} of $f$.
Since $\eta$ is non-zero on $\H$, 
the eta quotient $f$ 
is holomorphic if and only if $f$ does not have any pole at the cusps of 
$\Gamma_0(N)$.

We call an eta quotient $f$ \emph{primitive} if 
there does not exist any other eta quotient $h$ and any $\nu\in\N$
such that $f(z)=h(\nu z)$ for all $z\in\H$.
Let $f$, $g$ and $h$ be nonconstant \hes on $\Gm_0(M)$ (i.~e. 
their levels divide $M$) 
such that 
$f=g\times h$. Then we say that $f$ is \emph{factorizable on} $\Gm_0(M)$. 
We call a holomorphic eta quotient $f$ of level $N$ \emph{quasi-irreducible} (resp. \emph{irreducible}),
if it is not factorizable on $\Gm_0(N)$ (resp. on $\Gm_0(M)$ for all multiples~$M$ of $N$).
Here, it is worth mentioning that the notions of irreducibility and quasi-irreducibility of holomorphic eta quotients are conjecturally equivalent (see \cite{B-three}).

We say that a holomorphic eta quotient is \emph{simple} if it is 
 both primitive and quasi-irreducible.
Such eta quotients were first considered by Zagier, who conjectured (see \cite{z}) that:
 \begin{quote}
  \emph{There are only finitely many simple holomorphic eta quotients
  of a given weight.
}
 \end{quote}
This conjecture was 
established by his student Mersmann in 
an 
excellent \emph{Diplomarbeit} \cite{c}.  The proof of this conjecture 
occupied  more than half of his 110 pages long thesis about which 
K\"ohler at  p.~117 in \cite{b} wrote:
 \begin{quote}
``$\hdots$ the proof is rather long and can hardly
be called lucid, although doubtlessly it is ingenious. We were not able to
simplify it sufficiently so that we could reasonably incorporate it into this
monograph.''
 \end{quote}
Motivated by the above paragraph in K\"ohler's book, here we simplify 
the proof of
Mersmann's theorem by incorporating a few new ideas into it.  We recall from \cite{B-four} that the following analog of Zagier's conjecture also holds:
\emph{There are only finitely many simple holomorphic eta quotients
  of a given level.}
In \cite{B-three}, we see an application of Mersmann's theorem together with its above analog
to the fundamental problem of determining 
irreducibility of holomorphic eta quotients. In his thesis, Mersmann also proved another conjecture of Zagier on the exhaustiveness of Zagier's list of 
simple holomorphic eta quotients of 
weight $1/2$. We give a short proof of the last result in~\cite{B-six}. 
We shall also see examples of simple holomorphic eta quotients of arbitrarily large levels in \cite{B-seven}.


\section{Notations and the basic facts}
By $\N$ we denote the set of positive integers.
We define the operation $\odot:\N\times\N\rightarrow\N$ by 
\begin{equation}
 d_1\odot d_2:=\frac{d_1d_2}{\gcd(d_1, d_2)^2}.
\label{24.1July}\end{equation}
For $N\in\N$, by $\D$ we denote the set of divisors of $N$.
For 
$d\in\D$, we say that $d$ \emph{exactly divides} $N$ and write $d\|N$ 
if $\gcd(d,N/d)=1$.
We denote the set of such divisors of $N$ by $\cl{E}_N$. 
It follows trivially
that $(\cl{E}_N,\odot)$ is a boolean group (i.~e. each element of  $\cl{E}_N$ is the inverse of itself) 
and that $\cl{E}_N$ acts on $\D$ by $\odot$.
For 
$X\in\ZD$, we define the eta quotient $\e^X$
   by
   \begin{equation}\label{3Jan15.1}
    \e^X:=\displaystyle{\prod_{d\in\D}\eta_d^{X_d}},
 \end{equation}
where $X_d$ is the 
value of $X$ 
at $d\in\D$ whereas $\e_d$ denotes the rescaling of $\e$ by $d$.
Clearly, the level of $\e^X$ divides $N$. In other words, $\e^X$ transforms like a modular form on $\GN$. 

An \emph{eta quotient 
on $\Gm_0(N)$} is
an eta quotient whose level divides $N$.
For $N, k\in\Z$, let
 $\hypertarget{etamer}{\bb{E}^!_{N,k}}$ $($resp. $\hypertarget{etahol}{\bb{E}_{N,k}})$ 
 be the set
of eta quotients $($resp. holomorphic eta quotients$)$ of weight $k/2$ on $\GN$.
For $n\in\cl{E}_N$,
we define the \emph{Atkin-Lehner map}
$\al_{n,N}:\bb{E}^!_{N,k}\rightarrow\bb{E}^!_{N,k}$
by
\begin{equation}
\al_{n,N}\Big{(}\prod_{d\in\D}\e_d^{X_d}\Big{)}:=\prod_{d\in\D}\e_{n\odot d}^{X_d}.\label{29.1Jul}
\end{equation} 
 Since $\cl{E}_N$ is a boolean group and since it acts on $\D$ by $\odot$,
 it follows trivially that the map $\al_{n,N}:\bb{E}^!_{N,k}\rightarrow\bb{E}^!_{N,k}$ is an involution.
It is easy to show that the above definition 
is compatible
with the usual definition (see \cite{al}) of 
 Atkin-Lehner involutions of modular forms on $\GN$  up to multiplication by a 
complex number (see the Preliminaries in \cite{B-two}).
So, 
if $f$ is an eta quotient on $\GN$ and $n\in\cl{E}_N$,
then $f$ is holomorphic
if and only if so is $\al_{n,N}(f)$.
In particular,  the involution $W_N:=\al_{N,N}$ from $\bb{E}^!_{N,k}$ to $\bb{E}^!_{N,k}$ is called \emph{Fricke involution}. 

Recall that a holomorphic eta quotient $f$ on $\GN$ is
an eta quotient on $\GN$ that does not have any poles at the cusps. Under the action of $\GN$ on $\PQ$
by M\"obius transformation, for 
$a,b\in\Z$ with $\gcd(a,b)=1$,
we have
\begin{equation}
[a:b]\hspace{.1cm}
{{{{\sim}}}}_{\hspace*{-.05cm}{{{{\GN}}}}}
\hspace{.08cm}[a':\gcd(N,b)]\label{19.05.2015} 
\end{equation}
for some $a'\in\Z$ which is coprime to $\gcd(N,b)$ (see \cite{ds}).
We identify $\PQ$ with $\Q\cup\{\infty\}$ via the canonical bijection that maps $[\alpha:\lambda]$ to 
$\alpha/\lambda$ if $\lambda\neq0$ and to $\infty$ if $\lambda=0$. 
For $s\in\Q\cup\{\infty\}$ and a weakly holomorphic modular form $f$ on $\GN$, the order of $f$ at the cusp $s$ of $\GN$ is 
the exponent of  \hyperlink{queue}{$q^{{1}/{w_s}}$} 
occurring with the first nonzero coefficient in the 
$q$-expansion of $f$
at the cusp $s$,
where $w_s$ is the width of the cusp $s$ (see \cite{ds}, \cite{ra}).
The following is a minimal set of representatives of the cusps of $\GN$ (see \cite{ds}, \cite{ymart}):
\begin{equation}
\cl{S}_N:=\Big{\{}\frac{a}{t}\in\Q\hspace{2.5pt}{\di}\hspace{2.5pt}t\in\cl{D}_N,\hspace{2pt} a\in\bb{Z}, 
 \hspace{2pt}\gcd(a,t)=1\Big{\}}/\sim\hspace{1.5pt},
\end{equation}
where $\dfrac{a}{t}\sim\dfrac{b}{t}$ if and only if $a\equiv b\pmod{\gcd(t,N/t)}$.
For $d\in\D$ and for $s=\dfrac{a}t\in\cl{S}_N$ with $\gcd(a,t)=1$, we have
\begin{equation}
 \ord_s(\e_d\hspace{1pt};\GN)= \frac{N\cdot\gcd(d,t)^2}{24\cdot d\cdot\gcd(t^2,N)}\in\frac1{24}\N 
\label{26.04.2015}\end{equation}
(see 
\cite{ymart}). 
It is easy to check the above inclusion  
when $N$ is a prime power. 
The general case 
follows by multiplicativity (see (\ref{27.04.2015}), (\ref{20.3Sept}) and (\ref{13May})).
It follows that for all $X\in\ZD$, we have
\begin{equation}
  \ord_s(\e^X\hspace{1pt};\GN)= \frac1{24}\sum_{d\in\D}\frac{N\cdot\gcd(d,t)^2}{d\cdot\gcd(t^2,N)}X_d\hspace{1.5pt}. 
\label{27.04}\end{equation}
 In particular, 
that implies
\begin{equation}
 \ord_{a/t}(\e^X\hspace{1pt};\GN)=\ord_{1/t}(\e^X\hspace{1pt};\GN)
\label{27.04.2015.1}\end{equation}
for all $t\in\D$ and for all the $\varphi(\gcd(t,N/t))$ inequivalent cusps of $\GN$
represented by rational numbers
of the form $\dfrac{a}{t}\in\cl{S}_N$ with $\gcd(a,t)=1$,
where $\varphi$ denotes Euler's totient function.

We define the \emph{\hypertarget{om}{order map}}
$\cl{O}_N:\ZD\rightarrow\frac1{24}\ZD$ of level $N$ as the map which sends $X\in\ZD$ to the ordered set of
orders of the eta quotient $\e^X$ at the cusps $\{1/t\}_{t\in\D}$ of $\GN$. Also, we define
\emph
{order matrix} $A_N\in\Z^{\D\times\D}$ of level $N$ by 
\begin{equation}
 A_N(t,d):=24\cdot\ord_{1/t}(\e_d\hspace{1pt};\GN) 
\label{27.04.2015}\end{equation}
for all $t,d\in\D$.
By linearity of the order map, we have 
\begin{equation}
\cl{O}_N(X)=\frac1{24}\cdot A_NX\hspace{1.5pt}. 
\label{28.04}\end{equation}
From (\ref{27.04.2015}) and (\ref{26.04.2015}), we note that the matrix 
$A_N$ is not symmetric. 
It would have been much easier for us to work with $A_N$
if it would have been 
symmetric 
(for example, see
Lemma~\ref{il} and its proof 
 in Section~\ref{symsec}).
 So, we define the \emph{symmetrized order matrix} $\A_N\in\Z^{\D\times\D}$
by  
\begin{equation}
 \widehat{A}_N(t,\un)=\gcd(t,N/t)\cdot A_N(t,\un)\hspace{6pt}\text{for all $t\in\D$},
\label{20.3Sept}\end{equation}
where $\A_N(t, \un)$ (resp. $A_N(t, \un)$) denotes the row of $\A_N$ (resp. $A_N$) indexed by $t\in\D$.
For example, for a prime power $p^n$, 
we have
\begin{equation}
\A_{p^n}=\begin{pmatrix}
 \vspace{5.8pt}p^n &p^{n-1} &p^{n-2} 
&\cdots  &p &1\\
\vspace{5.8pt} p^{n-1} &p^n &p^{n-1} 
&\cdots  &p^2 &p\\
 \vspace{5.8pt}p^{n-2} &p^{n-1} &p^n 
&\cdots  &p^3 &p^2\\
\vspace{5.8pt} \vdots &\vdots &\vdots 
&\cdots &\vdots &\vdots\\
\vspace{5.8pt} p &p^2 &p^3 
&\cdots &p^n &p^{n-1}\\
 1 &p &p^2 
&\cdots  &p^{n-1} &p^n
\end{pmatrix}.
 \label{23July}
\end{equation}
For $r\in\N$, if $Y,Y'\in\Z^{\D^{\hspace{.5pt}r}}$ is 
such that $Y-Y'$ is
nonnegative at each element of $\D^{\hspace{.5pt}r}$, then
 we write $Y\geq Y'$. Otherwise, we write $Y\ngeq Y'$.
In particular, for $X\in\ZD$, the eta quotient $\e^X$ is holomorphic if and only if $\A_NX\geq0$.

From $(\ref{20.3Sept})$,
$(\ref{27.04.2015})$ and $(\ref{26.04.2015})$, we note that $\A_N(t,d)$ is multiplicative in $N$ and in $d,t\in\D$.
Hence, it follows that
\begin{equation}
 \A_N=\bigotimes_{\substack{p^n\|N\\\text{$p$ prime}}}\A_{p^n},
\label{13May}\end{equation}
where by \hspace{1.5pt}$\otimes$\hspace{.2pt},  we denote the Kronecker product of matrices.\footnote{Kronecker product of matrices is
not commutative. However, 
since any given ordering of the primes dividing $N$ induces a lexicographic ordering on $\cl{D}_N$ 
with which the entries of $\A_N$ are indexed, 
Equation (\ref{13May}) makes sense for all possible 
orderings of the primes dividing $N$.} 

It is easy to verify that for a prime power $p^n$, 
the matrix $\A_{p^n}$ is invertible with the tridiagonal inverse: 
\begin{equation}
\A_{p^n}^{-1}=
\frac{1}{p^n(1-\frac{1}{p^2})}
\begin{pmatrix}
\hspace{6pt}1 &\hspace*{-6pt}-\frac{1}{p} &  
& &  & \\
\vspace{5pt}\hspace*{-1pt}-\frac{1}{p} & \hspace*{-2pt}1+\frac{1}{p^2} &\hspace*{-6pt}-\frac{1}{p} 
& &\textnormal{\Huge 0} & \\
\vspace{5pt}&\hspace*{-6pt}-\frac{1}{p} & \hspace*{-4pt}1+\frac{1}{p^2} &\hspace*{-6pt}-\frac{1}{p} 
 &  &  \\
 &  & \ddots 
 & \ddots & \ddots & \\
\vspace{5pt}\hspace{2pt} &\textnormal{\Huge 0}  &  
 &\hspace*{-6pt}-\frac{1}{p} & \hspace*{-4pt}1+\frac{1}{p^2} &\hspace*{-6pt}-\frac{1}{p}\hspace{2pt}\\
\hspace{2pt} &  & 
&  &\hspace*{-6pt}-\frac{1}{p} & 1\hspace{2pt}
\end{pmatrix}.
 \label{r1}\end{equation}
For general $N$, the invertibility of the matrix $\A_N$ now 
follows by (\ref{13May}).
Hence, any eta quotient on $\GN$ is uniquely determined by its orders at the set of the cusps
$\{1/t\}_{t\in\D}$ of $\GN$. In particular, for distinct $X,X'\in\ZD$, we have $\e^X\neq\e^{X'}$. 
The last statement 
is also implied by the uniqueness of $q$-series expansion:
Let $\e^{\widehat{X}}$ and $\e^{\widehat{X}'}$
be the \emph{eta products} (i.~e.  $\widehat{X}, \widehat{X}'\geq0$)
obtained by multiplying $\e^X$ and $\e^{X'}$ with a common denominator. The claim follows by induction on the weight of $\e^{\widehat{X}}$
(or equivalently, the weight of $\e^{\widehat{X}'}$)
when we compare
the corresponding 
first two exponents of $q$
occurring in the $q$-series expansions of 
$\e^{\widehat{X}}$ and $\e^{\widehat{X}'}$.

For $r\in\N$ and $Y\in\Z^{\D^{\hspace{.5pt}r}}$, we define $|Y|\in\Z^{\D^{\hspace{.5pt}r}}$ 
as the integer-valued function on $\D^{\hspace{.5pt}r}$, whose value at each element of
$\D^{\hspace{.5pt}r}$ is the absolute value of the value taken by $Y$ at that point.
We define \hspace{1pt}$\mathds{1}_N,\hspace{.5pt}\lp_N$ and \hspace{1.5pt}$\bt_N\in\Q^{\cl{D}_N}$ by 
\begin{equation}
\mathds{1}_N(t):=1\hspace{5pt}\text{ for all \hspace{1pt}$t\in\cl{D}_N$}, 
\end{equation}
\begin{equation}
   \lp_N:=(\A_N^{-1})^\T\hspace*{1.5pt}\mathds{1}_N
   \hspace{3pt}\text{ and }\hspace{4.5pt}\bt_N:=|\A_N^{-1}|^\T\hspace*{1.5pt}\mathds{1}_N\hspace*{1.5pt}
. 
   \label{b00}\end{equation}
Then we have
\begin{equation}
   \lp_N(t)=\frac{\phi(t')}{
   t'\hspace{.5pt}\psi(N)}\hspace{4.5pt}\text{ and }\hspace{5.5pt}\bt_N(t)=\frac{\psi(t')}{
   t'\hspace{.5pt}\phi(N)}\hspace{5.5pt}\text{ for all \hspace{2pt}$t\in\cl{D}_N$},
   \label{b}\end{equation}
   where $t'=\gcd(t,N/t)$ and $\phi$, $\hspace{.5pt}\psi:\N\rightarrow\N$ are Euler's $\phi$ 
function 
and Dedekind's $\psi$ function given by
\begin{equation}
  \phi(N) := N\prod_{\substack{p|N\\\text{$p$ prime}}}\left(1-\frac{1}{p}\right)\hspace{5pt}\text{ and }\hspace{5pt}\psi(N):= N\prod_{\substack{p|N\\\text{$p$ prime}}}\left(1+\frac{1}{p}\right).
\label{14.2Sept}\end{equation}
Using (\ref{r1}), the assertions in (\ref{b}) are easy to check if $N$ is a prime power.
The general case again 
follows by multiplicativity.

Let $X\in\ZD$ and let
$f=\e^X$ be an eta quotient on $\GN$ of weight $k/2$ for some $k\in\Z$.
We recall the linear relation between $X$ and the orders of $f$
from (\ref{28.04}). Since $\A_N$ is invertible, so is $A_N$.
Since $\A_N$ 
is symmetric and since for $X\in\ZD$,
the weight the eta quotient $\e^X$ is equal to $\frac12\sum_{d\in\D}X_d$,
summing over the rows of $A^{-1}_N$,
from (\ref{28.04}),
(\ref{20.3Sept}), (\ref{b00}) and (\ref{b}), we get:
\begin{equation}
 \sum_{t\in\D}{\varphi(\gcd(t,N/t))}\cdot\ord_{1/t}(f\hspace{1pt};\GN)=\frac{k\cdot\psi(N)}{24}.
\label{27.04.2015.3}\end{equation}
The above equation is just a special case of the valence formula (see \cite{B-three}).
Since $\ord_{1/t}(f\hspace{1pt};\GN)\in\frac{1}{24}\Z$ (see (\ref{26.04.2015})), from
(\ref{27.04.2015.3}) it follows that of any particular weight, there are only finitely many holomorphic eta quotients
on $\GN$. More precisely, the number of holomorphic eta quotients of weight $k/2$
on $\GN$ is at most the number of solutions of the following equation
\begin{equation}
 \sum_{t\in\D}{\varphi(\gcd(t,N/t))}\cdot x_t=k\cdot\psi(N)
\label{27.04.2015.3Z}\end{equation}
in nonnegative integers $x_t$. Also, Corollary~\ref{4.3Aug} in the next section implies an- other upper bound on 
 the number of such 
 eta quotients.

We end this section with a set of notations which we shall use later:
\vspace*{-.4cm}

\
 \newline$X_{+}\hspace*{9pt}:\hs$ \hspace{1pt}the positive component of $X\in\R^{\D}$\hspace{-2pt},  i.~e.  $X_{+}=\frac{1}{2}(X+|X|)$.
 \newline$X_{-} \hspace*{9pt}:\hs$ \hspace{1pt}the 
negative component of $X\in\R^{\D}$\hspace*{-2pt},  i.~e.  $X_{-}=\frac{1}{2}(X-|X|)$.
 \newline$\sigma(X) \hspace*{.5pt}:\hs$ \hspace{1pt}the sum of the values of $X\in\R^{\D}$. So,  $\sigma(X)=\sigma(X_{+})+\sigma(X_{-})$.
 \newline$\|X\| 
 \hspace{5pt}:\hs$  \hspace{1pt}the $L^1$ norm of $X\in\R^{\D}$,
i.~e.  $\|X\|= \sigma(|X|)=
 \sigma(X_{+})-\sigma(X_{-})$.
 \newline$\|X\|_{\pm}\hspace*{-2pt}:\hs$ \hspace{1pt}
the 
$L^1$ norm of $X_{\pm}\in\R^{\D}$, i.~e. $\|X\|_{\pm}=\|X_{\pm}\|=\pm\sigma(X_{\pm}).$
\vspace*{4.5pt}

\section{Generalization of a result of Mersmann / Rouse-Webb} 
The following lemma will be crucial in our proof of Mersmann's theorem: 
 \begin{lem}For $X\in\ZD$, let $\e^X$ be an eta quotient of weight $k/2$ on $\GN$.
Then we have
\begin{equation}
 \|X\|\leq k\cdot F(N)+G(N)\cdot\|\widehat{A}_NX\|_{-},
\label{akissintherain}
\end{equation}
where 
\begin{equation}
 F(N):= \frac{\psi(N)}{\varphi(N)}\hs\hs\cdot\hspace*{-.1cm}\prod_{\substack{\text{$p$ prime}\\p^2\idi N}}\frac{p+1}{p-1}
 \hs\hs,\hspace{4pt}G(N):=\Bigg{(}\frac1{\psi(N)}+\frac1{\phi(N)}\Bigg{)}\hs\cdot\hspace*{-.1cm}\prod_{\substack{\text{$p$ prime}\\p^2\idi N}}\Big{(}1+\frac1p\Big{)}
\nonumber\end{equation}
and $\psi$, $\phi:\N\rightarrow\N$ are as defined in $(\ref{14.2Sept})$.
\label{4.3Aug}\end{lem}
We recall that for $X\in\ZD$, $\e^X$ is holomorphic if and only if $\A_NX\geq0$.
So, 
from 
the above lemma, we obtain: 
\begin{co}[Mersmann / Rouse-Webb]
For $X\in\ZD$, let $\e^X$ be a holomorphic eta quotient of weight $k/2$ on $\GN$. 
Then we have\label{20.2Sept}
\begin{equation}\|X\|\hspace{1pt}\leq\hspace{1pt} 
k\hspace{1pt}F(N).\label{ipll}\end{equation}
\end{co}
\begin{asi}
 \textnormal{
The 
last inequality implies that 
the number of holomorphic eta quotients of weight $k/2$ on $\GN$ 
is less than $(2kF(N))^{\operatorname{d}(N)}$,
where $\operatorname{d}(N)$ denotes the number of divisors of $N$.
 But the dimension of the space of modular forms of any fixed even weight on $\GN$
 becomes arbitrarily large as $N\rightarrow\infty$ (see \cite{ds}). So, if we fix 
 the number of divisors of 
 $N$ along with a \hs$k\in4\N$, then except only finitely many possibilities for 
 $N$,
 the space of modular forms of weight $k/2$ 
 on $\GN$ never contains enough eta quotients to constitute
  a basis. This gives a partial answer to a question asked by Ono in \cite{on} about classification
 of the spaces of modular forms which are spanned by eta quotients.
}
 \end{asi}

\begin{re}\textnormal{Mersmann actually proved a variant of the above corollary in 
1991. 
As a part of my doctoral research,
I proved Lemma~\ref{4.3Aug} and obtained the above consequences in 2011 (and presented them 
 at the 1st EU-US Conference on Automorphic Forms and Related Topics
at Aachen in 2012). Independent of both Mersmann's and my earlier works (see \cite{c} and Chapter~3 in \cite{B-two}),
Rouse and Webb also proved the same statement as of Corollary~\ref{20.2Sept} (see Theorem~2 in \cite{rw}) 
in 2013
and drew a similar conclusion as above on the spaces of modular forms being spanned by
eta quotients (in fact, they studied the spaces spanned by eta quotients in much greater details in \cite{rw}).
}
\end{re}

\begin{co}
Let $f=\e^X$ be a weakly holomorphic eta quotient of weight $k/2$ on $\GN$
with $\|X\|\geq k\hspace{1pt}F(N)+\varepsilon$. Then we have 
\begin{equation}
\|\widehat{A}_NX\|_{-}\hspace{.5pt}\geq\varepsilon/G(N)\hspace{1pt}.\label{inik}
\end{equation}
\end{co}
\begin{co}
Let $f=\e^X\neq1$ be a weakly holomorphic eta quotient of weight $k/2$ 
on $\GN$, where $k\leq0$.
Then
we have
 \begin{equation}
\|\widehat{A}_NX\|_{-}\hspace{.5pt}\geq\left\{\begin{array}{ll}
                                               2/G(N)&\text{if $k=0$}\\
                                               |k|\cdot(F(N)+1)/G(N)&\text{otherwise.}
                                              \end{array}\right.
\label{in}
                \end{equation}                                                                                                      
\end{co}
\begin{proof}We have $\|X\|\geq|\sigma(X)|=|k|$. 
Since $\e^X\neq1$, 
$X\neq0$.
So, if
$\sigma(X)=k=0$, 
then $X\in\ZD$ has at least two nonzero entries. 
Hence, we have $\|X\|\geq\max\{2,|k|\}$ for all $k$. 
The claim now follows from (\ref{akissintherain}).
\end{proof}

\ \newline
\textit{Proof of Lemma~$\ref{4.3Aug}$}. 
 Let $Y:=\A_NX$. Then we have
     \begin{equation}\|X\|\hspace{.5pt}=\hspace{.5pt}\|\A_N^{-1}Y\|\hspace{1pt}\leq\hspace{1pt}
     \|\A_N^{-1}Y_{+}\|+\|\A_N^{-1}Y_{-}\|\hspace{.5pt}.\label{da1}\end{equation}
We define the set $\cl{Q}_N\subset\D$ by 
$\cl{Q}_N:=\{d\hspace{2.5pt}\idi\hspace{2.5pt}d^2\in\D\}$.
Then for all $t\in\D$, we have $(t,N/t)\in\cl{Q}_N$
and for all $d\in\cl{Q}_N$, we have $(d,N/d)=d$. 
For $d\in\cl{Q}_N$, let 
\begin{equation}
y_d:=\sum_{\substack{t\in\D\\Y_t<0\\(t,N/t)=d}}|Y_t|\hspace{.5pt}.\end{equation}
Since $\A_N$ is symmetric, from (\ref{b00}) and (\ref{b}), we get 
\begin{equation}
      \mathds{1}_N^\T\A_N^{-1}Y_{-}=-\sum_{d\in\cl{Q}_N}y_d\frac{\varphi(d)}{d\psi(N)} 
     \label{da2A}\end{equation}
and
 \begin{equation}     
      \|\A_N^{-1}Y_{-}\|\hspace{1pt}\leq\hspace{1pt}\mathds{1}_N^\T
      \hspace{1pt}|\A_N^{-1}|\hspace{2pt}|Y_{-}|=\sum_{d\in\cl{Q}_N}y_d\frac{\psi(d)}{d\varphi(N)}\hspace{1.3pt}.
     \label{da2}\end{equation}
Again from (\ref{b00}), (\ref{b}) and (\ref{da2A}),  we obtain 
\begin{align}
\dfrac1{\psi(N)}\Big{(}\min_{d\in\cl{Q}_N}\dfrac{\phi(d)}{d}\Big{)}\hs\mathds{1}_N^\T Y_{+}\hspace{1pt}&\leq\hspace{1pt}\mathds{1}_N^\T \A_N^{-1}Y_{+}=k-\mathds{1}_N^\T \A_N^{-1}Y_{-}\nonumber\\
&=k\hs+\sum_{d\in\cl{Q}_N}y_d\frac{\varphi(d)}{d\psi(N)}\hspace{2pt}.       
      \label{da3}\end{align}
\vspace*{-.2cm}It follows that
\begin{align}
 \|\A_N^{-1}Y_{+}\|\hspace{1.5pt}&\leq\hspace{1.5pt}\mathds{1}_N^\T\hspace{1pt}|\A_N^{-1}|\hspace{1pt}Y_{+}\hspace{1.5pt}\leq
 \hspace{1.5pt}
 \dfrac1{\phi(N)}\Big{(}\max_{d\in\cl{Q}_N}\dfrac{\psi(d)}{d}\Big{)}\hs\mathds{1}_N^\T Y_{+}\nonumber\\ 
 &\leq \frac{\psi(N)}{\phi(N)}\Big{(}\max_{d\in\cl{Q}_N}\dfrac{\psi(d)}{\phi(d)}\Big{)}\Bigg{(}k\hs+\sum_{d\in\cl{Q}_N}y_d\frac{\varphi(d)}{d\psi(N)}\Bigg{)}
\label{da4}\end{align}
where the first inequality is trivial, the second 
follows from (\ref{b00}) and (\ref{b}),
whereas  
and the third inequality follows from (\ref{da3})
and from the fact that $\psi(d)/d$ and $\phi(d)/d$ attain
respectively the maximum and minimum for the same values of $d$ in $\cl{Q}_N$.

Now, from (\ref{da1}), (\ref{da2}) and (\ref{da4}),
we get
\begin{align}
 \|X\| 
&\leq  k\cdot F(N)
+\Big{(}\max_{d\in\cl{Q}_N}\dfrac{\psi(d)}{d}\Big{)}\Bigg{(}\frac1{\phi(N)}+\frac1{\psi(N)}\Bigg{)}\cdot\sum_{d\in\cl{Q}_N}y_d\\
%
&\leq k\cdot F(N)+G(N)\cdot\|\A_NX\|_{-}\hs.\nonumber\end{align}
\qed

\section{Proof of the finiteness}
Mersmann's finiteness theorem
follows 
from (\ref{27.04.2015.3}) or (\ref{ipll}), if for any given $k\in\N$
the existence of a simple holomorphic eta quotient of weight $k/2$ and level $N$
implies only finitely many possibilities for $N$. Below we show that this is indeed true:
\vspace{3pt}

 Let 
 $N=P_1^{r_1}P_2^{r_2}\cdots P_m^{r_m}$, where $P_1<P_2<\cdots<P_m$ are primes. 
For any $d\|N$, there exists a canonical bijection between $\bb{Z}^{\D}$ and $\bb{Z}^{\cl{D}_{{N/d}}\times\cl{D}_d}$. 
If $d=P_i^{r_i}$, then we denote the image of $X\in\bb{Z}^{\D}$ by $X^{(i)}$ under this bijection,  i.~e.
if we set $N_i:=\frac{N}{P_i^{r_i}}$ then 
\begin{equation}
X_{\nu,\hspace{.5pt}P_i^j}^{(i)}=X_{\nu P_i^j}\hspace{2pt},\hspace{3pt} \ \nu|N_i\hspace{5pt}\textnormal{and}\hspace{5pt}0\leq j\leq r_i\hspace{2pt}.
\label{swargosondhyan}\end{equation}
For any nonnegative integer $j\leq r_i$, we call $X^{(i)}_j:=\{X_{\nu,\hspace{.5pt} P_i^{j}}\}_{\nu\hspace{.5pt}\in D(N_i)}$ 
the \emph{$j$-th column} of $X^{(i)}$.
Let
$F_m:=F(p^2_1p^2_2\cdots p^2_m)$, where $p_1=2$, $p_2=3,\hdots, p_m$ are the first $m$ primes
and the function $F$ 
is as defined in Lemma~\ref{4.3Aug}.
 It is easy to note that $F(N)\leq F_m$ and 
 from Mertens' theorem (see 
 \cite{e}), it follows that   
 that $F_m=O(\log^4 m)$ as $m\rightarrow\infty$. Later in this section, we shall show that for $k\in\N$,
 there exists a constant $C_k$ such that 
 if $\e^X$ is a simple holomorphic eta quotient of level $N$ and weight $k/2$, then
 for all primes $P_i\idi N$, 
 we have
\begin{equation}
 P_i^{\delta_i}\leq C_kF(N)(F(N)+1)^2,
\label{koelnsued}\end{equation}
where 
$\delta_i:= 1\hspace{2pt}+$ the highest number of consecutive zero columns in $X^{(i)}$.
In particular, from (\ref{koelnsued}) for $i=m$ we get $P_m=O_k(\log^{12}m)$. Since the Prime Number Theorem (see \cite{e}) 
implies that $P_m\geq p_m\sim m\log m$, 
Inequality (\ref{koelnsued}) puts a bound on 
$m$ as well as on all primes $P_i|N$.\hspace{2pt}\footnote{\hspace{1.5pt} 
In general, this 
naive bound is 
very large.
For example, we have $C_1={1}/{4}$. The order of magnitude of the largest prime $p_m$ for which the 
inequation $p_m<F_m(F_m+1)^2/4$ holds is $10^{18}$.
Whereas actually, 
the greatest prime divisor of the level of a 
simple holomorphic eta quotient of weight $1/2$ is at most 3 (see \cite{z}, \cite{c} or \cite{B-six}).}

For each $i$, $X^{(i)}$ has $r_i+1$ columns. 
Since $\eta^X$ is primitive, the first column of
$X^{(i)}$ is nonzero 
and since $\e^X$ is of level $N$, the last column of $X^{(i)}$ is nonzero. 
Therefore,
the number of nonzero columns in $X^{(i)}$ is at 
least $\frac{r_i}{\delta_i}+1$.
Hence, from (\ref{ipll})
we get
\begin{equation}
 \frac{r_i}{\delta_i}+1\leq k\hs F(N). 
\label{Enumerable: the sets whose elements are easy to recognise if given, but there may not exist an algorithm to find.}\end{equation}
Since $(\ref{koelnsued})$ and $(\ref{Enumerable: the sets whose elements are easy to recognise if given, but there may not exist an algorithm to find.})$
together 
impose a bound on $r_i$, we have only finitely many possibilities for $N$ if $k$ is given.\qed


Now, we construct a decreasing function $g:\bb{Z}\rightarrow\Q_{>0}$
 such that
if  
$\eta^X$  is a simple holomorphic 
eta quotient of weight $k/2$ and level $N$, then 
for any prime $P_i|N$, $(\ref{koelnsued})$ is satisfied if we put 
\begin{equation}C_k=\frac{k}{\hspace{1pt}2g(k-1)}\hspace{2pt}.\label{saint0}\end{equation}                                                             
For all $n<0$, we set $g(n)=2\cdot|n|$ and we set $g(0)=2$. 
For $n>0$, below we define $g(n)$ inductively. 
Let $G$ be the function as defined in Lemma~\ref{4.3Aug}.
Since $G(N)\rightarrow0$ as $N\rightarrow\infty$, for all sufficiently large $M\in\N$, we have 
\begin{equation}
G(M)<g(n-1)\hspace{2pt}. 
\label{saint1}\end{equation}
Let $M_n$ be the least positive integer $M$ for which (\ref{saint1}) holds and let 
\begin{equation}
c_n:=G(M_n).
\label{saint1A}
\end{equation}
As before, Mertens' theorem and Prime Number Theorem (see \cite{e}) together imply that
there are only finitely many $M\in\N$ 
such that for each prime power $p^{r}\|M$, we have
\begin{equation}
\frac{M^2G(M)}{2p^{r_n(M)}\varphi(M)}(n(F(M)+1)+c_n)\geq g(n-1)-c_n\hspace{1.5pt},
\label{saint2}\end{equation}
where 
\begin{equation} 
r_n(M):=\frac{r}{n\hs F(M)+c_n-1}.
\nonumber\end{equation}                                                                                    
Let $M'_n$ be the greatest positive integer $M$ for which (\ref{saint2}) holds and let
\begin{equation}
\label{saint2A}N_n:=\max\{M_n,M_n^\prime\}.
\end{equation}
We define 
\begin{equation}
g(n):=\min_{M_n\leq M\leq N_n} G(M).
\label{saint3}\end{equation}

\

In order to prove (\ref{koelnsued}), we require the following lemmas:
\begin{lem}
For  $N\in\N$, $X\in\bb{Z}^{\D}$ and any prime power $P_i^{r_i}\|N$, we have\hspace{1pt}\footnote{\hspace{1.5pt}In fact, for any $d\|N$, if we denote by $X^ {[d]}$ the 
image
of $X\in\bb{Z}^{\D}$ in $\bb{Z}^{\cl{D}_{{N/d}}\times\cl{D}_d}$, 
then we have $(\A_NX)^{[d]}=\A_{{N/d}}X^{[d]}\A_d$\hspace{1pt} (see Section~4 in \cite{B-six}).}
$$(\A_NX)^{(i)}=
{\A_{N_i}X^{(i)}\A_{P_i^{r_i}}},$$
where $(\A_NX)^{(i)}$ and $X^{(i)}$ are defined similarly as in $(\ref{swargosondhyan})$. 
\label{il}\end{lem}
For $N, P_i$ and $r_i$ as above, for $X\in\ZD$ and for a real interval $I\subseteq[0,r_i]$, 
by $X^{(i)}_{I}$ we denote 
the submatrix of $X^{(i)}$ consisting of successive columns of it with indices in $I$.
For $N\in\N$ and $M\in\D$,  by $\pi_{\nm}:\bb{Z}^{\D}\rightarrow\vspace{1.5pt}\bb{Z}^{\cl{D}_M}$ we denote 
the projection map, i.~e. 
\begin{equation}
 \pi_{\nm}(X)_d=X_d \hspace{1.5pt}\text{ for all $X\in\ZD$ and $d\in\cl{D}_M$.}
 \nonumber\end{equation}
\begin{lem}
Let $N, P_i$ and $r_i$ as above. Let $X\in\ZD$ 
and let
$a,b\in\bb{Z}$ with  $0\leq a< b\leq r_i$ 
such that $X^{(i)}_{(a,b)}\in\emptyset$ $(\textnormal{i.~e. \ } b=a+1)$ or $X^{(i)}_{(a,b)}=0$.
Let $N':=P_i^aN_i$.
Then\label{mark} $$
\|\A_NX\|_-\geq P_i^{r_i-a}\cdot\left\|\A_ {N'}\hspace{1.5pt}\pi_{N,N'}(X)\right\|_-
-\frac{N^2}{P_i^{b-a}\phi(N)}\cdot\|X\|_{+}\hspace{.75pt}.$$
\end{lem}

\begin{lem}
For $N\in\N$ and  $X\in\ZD$, if $\A_NX\ngeq0$, then\label{mx}
$$
\|\A_NX\|_-\geq\dfrac{g(\sigma(X))}{G(N)}.$$
\end{lem}

We shall prove these lemmas
in the next section.
 Let $s_{\mn}:\bb{Z}^{\cl{D}_M}\rightarrow\bb{Z}^{\D}$ be a section of the projection $\pi_{\nm}$ 
such that for all 
$X\in\bb{Z}^{\cl{D}_M}$ and $d\in\D$, we have 
\begin{equation}
 s_{\mn}(X)_d=\begin{cases}X_d&\text{if $d\idi M$}\\
              0  &\text{otherwise.}
  \end{cases}
\nonumber\end{equation}
We define\hspace{1pt} $\widetilde{s}_{\mn}:= s_{\mn}\circ\pi_{\nm}$.

\ \newline\textit{Proof of (\ref{koelnsued}).} Let $a,b\in\bb{Z}$ with  $0\leq a< b\leq r_i$ such that $X^{(i)}_{(a,b)}=0$ and $a-b=\delta_i$. 
For ease of notation, we write 
$p=P_i$, $N'=p^aN_i$ and $r=r_i$.
Since $\eta^X$ is primitive, 
$\widetilde{s}_{{N',N}}(X)\neq0$ and since $\e^X$ is of level $N$, 
$\widetilde{s}_{{N',N}}(X)\neq X$. 
Let $k_1:=\sigma(\widetilde{s}_{{N',N}}(X))$
and $k_2:=\sigma(X-\widetilde{s}_{{N',N}}(X))$. 
So, $k_1+k_2=k$.
Now, if $k_1$ or \vspace{1pt} $k_2\leq0$, then by (\ref{in}), we have respectively $\A_N\widetilde{s}_{{N',N}}(X)\ngeq0$
or 
$\A_N(X-\widetilde{s}_{{N',N}}(X))\ngeq0$. Otherwise,\vspace{1pt} $0<k_1,\hspace{1pt}k_2<k$.
Since $\eta^X$ is irreducible, we still have either  $\A_N\widetilde{s}_{{N',N}}(X)\ngeq0$ or $\A_N(X-\widetilde{s}_{{N',N}}(X))\ngeq0$. 
Therefore if necessary, replacing $X$ by $\widetilde{X}$ 
where $\e^{\widetilde{X}}=\al_{n,N}(\e^X)$ for some $n\in\cl{E}_N$ with $p\idi n$
(hence, replacing $a$ by $r-b$ and $b$ by $r-a$), 
we may assume that $k_1<k$ and \begin{equation}\A_{{N}}\widetilde{s}_{{N',N}}(X)\ngeq0\hspace{2pt}.\label{somik1}\end{equation}  
We have \begin{equation}
\widetilde{s}_{{N',N}}(X)^{(i)} \A_{p^{r}}=\left(\pi_{{N,N'}}(X)^{(i)}\hspace{1pt}\big{|}\hspace{1pt}0\hspace{1pt}\right)
\left(\begin{array}{c|c}
p^{r-a}\A_{p^{a}}&B\\
\hline
\text{0}&\text{0}\\
\end{array}\right)       
\label{somik2}
\nonumber\end{equation}
where the $j$-th column of $B=p^{r-a-j}$( the last column of $\A_{p^{a}}$), for all $j\leq r-a$. Hence from (\ref{somik1}), via 
Lemma \ref{il} we get
\begin{equation}
 \A_{{{N'}}}\pi_{{N,N'}}(X)\ngeq0\hspace{2pt}.
\label{somik3}\end{equation}
Since 
$\|X\|_{+}=
(\|X\|+k)/2$, 
from Lemma \ref{mark} we have 
\begin{equation}
 \|\A_{{N}}X\|_-\geq p^{r-a}\|A_ {{N'}}\hspace{1pt}\pi_{{N,N'}}(X)\|_-
-\frac{N^2}{2p^{b-a}\phi(N)}(\|X\|+k)\hspace{2pt}.
\label{fartig0}\end{equation}
Since $\e^X$ is holomorphic, 
$\|\A_{{N}}X\|_-=0$. So, from (\ref{fartig0}), (\ref{somik3}), Lemma \ref{mx} and 
(\ref{ipll})
, we get 
\begin{equation}
\frac{kN^2}{2p^{b-a}\phi(N)}(F(N)+1)\geq p^{r-a}\frac{g(k_1)}{G(N')}\geq\frac{g(k-1)}{G(N)},\label{fartig1}\end{equation}
where the last inequality holds since $g$ is a decreasing function and since 
$p^{r-a}G(N)\geq G(N')$.
It follows trivially from the definitions of $F$ and $G$ that
\begin{equation}
 \frac{N^2G(N)}{\phi(N)}\leq F(N)(F(N)+1).
\label{fartig2}\end{equation}
Now, Inequality (\ref{koelnsued}) follows from (\ref{fartig1}), (\ref{fartig2}) and (\ref{saint0}).\qed
\section{Proofs of the lemmas}
\label{symsec}
\begin{flushleft}\textit{Proof of Lemma \ref{il}.} 
Follows from the facts that $\A_{{N}}=\A_{{P_i^{r_i}}}\otimes \A_{{N_i}}$\hspace{1pt} and that these matrices  are symmetric (see Lemma 4.3.1 in \cite{a}).\qed 
\end{flushleft}
\  \newline \textit{Proof of Lemma \ref{mark}.} 
To lighten the notation, we write $p=P_i$ and $r=r_i$. From Lemma~\ref{il}, we have 
\begin{equation}
 (\A_{N}X)^{(i)}=\A_{{N_i}}\widetilde{s}_{{N',N}}(X)^{(i)}\A_{p^e}+\A_{{N_i}}\left(X^{(i)}-\widetilde{s}_{{N',N}}(X)^{(i)}\right)\A_{p^e}\hspace{1pt}.
\label{ujbuk}\end{equation}
Therefore,
\begin{align}
\widetilde{s}_{{N',N}}(\A_{N}X)^{(i)}&=\Big{(}\pi_{{N,N'}}(\A_{N}X)^{(i)}\hspace{1pt}\big{|}\hspace{1pt}0\Big{)}
=\Big{(}(\A_{N}X)^{(i)}_{[0,a]}\hspace{1pt}\big{|}\hspace{1pt}0\Big{)}\nonumber\\
&=\A_{{N_i}}\Big{(}\pi_{{N,N'}}(X)^{(i)}\hspace{1pt}\big{|}\hspace{1pt}0\Big{)}\left(\begin{array}{c|c}
p^{r-a}\A_{p^{a}}&\text{0}\\
\hline
\text{0}&\text{0}\\
\end{array}\right)\label{bubblegum}\\
&\quad\qquad\quad+
\A_{{N_i}}\Big{(}X^{(i)}-\widetilde{s}_{{N',N}}(X)^{(i)}\Big{)}
\left(\begin{array}{c|c}
\text{\Large 0}&\text{\Large 0}\\
\hline
p^{r-|j-\ell|}_{\quad\substack{\\ \\ b\leq j\leq r \\ \hspace*{3pt}0\leq\ell\leq a}} &\hspace{1pt}_\text{\Large 0}\\
\end{array}\right),
\nonumber\end{align}
where the first two equalities are trivial and the third follows from (\ref{ujbuk}) and (\ref{23July}).
By Lemma \ref{il}, the absolute value of the sum of the negative entries 
in the 1st term of (\ref{bubblegum}) is\hspace{1pt} $p^{r-a}\|\A_{N'}\pi_{{N,N'}}(X)\|_-\hspace{1pt}$. 
The\vspace{3pt} sum of positive entries in the 2nd term of (\ref{bubblegum}) is less than or equal to 
$$\mathds{1}_{{N_i}}^\T
\A_{{N_i}} 
X_{+}^{(i)} 
\left(\begin{array}{c|c}
\text{\Large 0}&\text{\Large 0}\\
\hline
p^{r-|j-\ell|}_{\quad\substack{\\ \\ b\leq j\leq r \\ \hspace*{3pt}0\leq\ell\leq a}} &\hspace{1pt}_\text{\Large 0}\\
\end{array}\right)\mathds{1}_{p^r}.$$
We have
\begin{equation}
\mathds{1}_{{N_i}}^\T\A_{{N_i}}\leq\frac{N_i^2}{\phi(N_i)}\mathds{1}_{{N_i}}^\T.
\end{equation}
The above inequality follows from (\ref{23July}) if $N_i$ is a prime power. The general case then
follows by multiplicativity (see (\ref{13May})).
Since
\begin{equation}
\left(\begin{array}{c|c}
\text{\Large 0}&\text{\Large 0}\\
\hline
p^{r-|j-\ell|}_{\quad\substack{\\ \\ b\leq j\leq r \\ \hspace*{3pt}0\leq \ell\leq a}} &\hspace{1pt}_\text{\Large 0}\\
\end{array}\right)\mathds{1}_{p^r}\leq\hspace{1pt} 
\frac{p^{r-(b-a)}}{1-\frac1p}\hspace{1pt}\mathds{1}_{p^r},
\end{equation}
the sum of the positive entries in the second term of (\ref{bubblegum}) is less than or equal to
\begin{equation}
 \frac{p^{r-(b-a)}N_i^2}{(1-\frac1p)\phi(N_i)}\hspace{1.5pt}\mathds{1}_{{N_i}}^\T X_{+}^{(i)}\mathds{1}_{p^r}=\frac{N^2}{p^{b-a}\phi(N)}\hspace{1pt}
 \|X\|_{+}\hspace{.75pt}. 
\nonumber\end{equation}
Thus, we have
\begin{align*}
 \|\A_{{N}}X\|_-\hspace{1.5pt}&\geq\hspace{1.5pt}\|\hspace{1pt}\widetilde{s}_{{N',N}}(\A_{{N}}X)\|_{-}\\
 &\geq\hspace{1.5pt} p^{r-a}\|\A_{{N'}}\pi_{{N,N'}}(X)\|_-
 -\frac{N^2}{p^{b-a}\phi(N)}\hspace{1pt}
 \|X\|_{+}\hspace{.75pt}.
\end{align*}
\qed

\ \newline \textit{Proof of Lemma \ref{mx}.} For $X\in\ZD$, we have $\A_NX\in\ZD$. 
So, the lemma holds trivially for $N=1$. We proceed by induction on $N$. Let $M>1$ be an integer 
and let us assume that the lemma holds for all $N<M$. Let $X\in\bb{Z}^{\cl{D}_M}$ such that $A_MX\ngeq0$. 
Let $n:=\sigma(X)$. If $n\leq0$, then 
the lemma holds by (\ref{in}). So, let us assume that $n>0$. 
By the definition of $g$, the lemma holds trivially for $M$ if $M\leq N_n$,
where $N_n$ is as defined in (\ref{saint2A}).
So, we may assume that \begin{equation}M>N_n.\end{equation}
Since $c_n\geq g(n)$ (see (\ref{saint1A}) and (\ref{saint3})),
if $\|X\|\geq nF(M)+c_n$, the claim holds by (\ref{inik}).
So, we may also assume that
\begin{equation}
 \|X\|< nF(M)+c_n.
\label{oldtimes}\end{equation}
Since $M>N_n\geq M_n^\prime$ (see (\ref{saint2A}) and (\ref{saint2})), there exists a prime $p$ dividing $M$ such that
\begin{equation}
\frac{M^2G(M)}{2p^{r_n(M)}\varphi(M)}(n(F(M)+1)+c_n)<g(n-1)-c_n,
\label{stts0}\end{equation}
where \begin{equation} 
r_n(M):=\frac{r}{n\hs F(M)+c_n-1},
\label{asyousow}\end{equation}   
where $r\in\N$ such that $p^{r}\|M$. 
Since $A_{{M}}X\ngeq0$, $X$ is nonzero. 
For $P_i:=p,$ $r_i:=r$ and $M_i:=M/p^r$,
we define $X^{(i)}\in\Z^{\cl{D}_{M_i}\times \cl{D}_{P^{r_i}_i}}$ by  (\ref{swargosondhyan}),
 replacing $N_i$ 
 with $M_i$ 
 in it.
 
First we consider the case where at least one of the columns $X^{(i)}_0$ or $X^{(i)}_r$ is entirely zero.
If   $X^{(i)}_r\neq0$ and $X^{(i)}_0=0$, then we replace $X$ by $\widetilde{X}$ (thereby interchanging the 
columns $X^{(i)}_0$ and $X^{(i)}_r$),
where $\e^{\widetilde{X}}$ is the image under the Fricke involution $W_{M}$.
As \hspace{1pt}$\A_{{M}}(M\odot t,M\odot d)=\A_{{M}}(t, d)$ for all $t,d\in\Z^{\cl{D}_M}\hspace*{-1.5pt},$  from the definition of Fricke involution (see (\ref{29.1Jul})), it follows indeed that
\begin{equation}
\|\A_{{M}}X\|_-=\hspace{3pt}\|\A_{{M}}\widetilde{X}\|_-\hspace{.75pt}.
\label{golden-afternoon}\end{equation}
So, we may assume that $X^{(i)}_r=0$. Let 
$a\in\N$ be
such that $X^{(i)}_a\neq0$  and $X^{(i)}_b=0$ for all $b$ with $a<b\leq r$. Let $M':=\frac{M}{p^{r-a}}$. Then we have \begin{equation}
\widetilde{s}_{{M',M}}(X)=X.
\label{sky-over-the-neon}\end{equation}
 Since $A_{{M}}\widetilde{s}_{{M',M}}(X)\ngeq0$, the same argument which led us from (\ref{somik1})
to (\ref{somik3}) also implies that 
 $A_{{M'}}\pi_{{M,M'}}(X)\ngeq0$. So, we obtain
\begin{equation}\|A_{{M}}X\|_-\geq\|\hspace{1pt}\widetilde{s}_{{M',M}}(A_{{M}}X)\|_-=p^{r-a}\|A_{{M'}}\hspace{1.5pt}\pi_{{M,M'}}(X)\|_-\hspace{2pt},
\label{late-in-library}
\end{equation}
 where 
 the last equality 
follows from (\ref{bubblegum}), (\ref{sky-over-the-neon}) and Lemma~\ref{il}.
Now, by the induction hypothesis, we get
\begin{equation}
p^{r-a}\|A_{{M'}}\hspace{1.5pt}\pi_{{M,M'}}(X)\|_-\geq
 p^{r-a}\frac{g(n)}{G(M')}\geq\frac{g(n)}{G(M)},
\label{koeln-messe-deutz}\end{equation}
where the first inequality holds since 
 $\sigma(\pi_{{M,M'}}(X))=\sigma(X)=n$ 
and the second inequality holds since $p^{r-a}G(M)\geq G(M')$.
  Thus, from (\ref{late-in-library}) and (\ref{koeln-messe-deutz}) the claim follows in this case.

Now, we consider the 
remaining case where neither of the columns $X^{(i)}_0$ and $X^{(i)}_r$ are entirely zero.
We choose $a,b\in\Z$ with $0\leq a<b\leq r$
such that neither $X^{(i)}_a$ nor $X^{(i)}_b$ is entirely zero but $X^{(i)}_j=0$ for all $j\in\Z\cap(a,b)$
and $b-a -1=$ the highest number of consecutive zero columns in $X^{(i)}$. 
Since $X^{(i)}$ has $r+1$ columns
and since none of its extremal columns are entirely zero, number of nonzero columns of $X^{(i)}$ is 
at least ${r}/{(b-a)}+1$. Hence, from (\ref{oldtimes}) we get
\begin{equation}
\frac{r}{b-a}+1\leq nF(M)+c_n.
\label{Itwouldbefinishedsoon}\end{equation}
Let $M':=\frac{M}{p^{r-a}}$ and $n_1:=\sigma(\pi_{{M,M'}}(X))$.
 Since $A_{{M}}X\ngeq0$,
either $A_{{M}}\hspace{1pt}\widetilde{s}_{{M',M}}(X)$ $\ngeq0$ 
or $A_{{M}}(X-\widetilde{s}_{{M',M}}(X))\ngeq0$. If necessary, replacing $X$ by $\widetilde{X}$ 
where 
$\e^{\widetilde{X}}$ is the image of $\e^X$ under the Fricke involution $W_N$ (see (\ref{golden-afternoon}))
and using 
an argument 
similar to what we used at the beginning of 
the proof of 
(\ref{koelnsued}) (see (\ref{somik1})),
we may assume that $n_1<n$ and $A_{{M}}\hspace{1pt}\widetilde{s}_{{M',M}}(X)\ngeq0$.
As before, that implies:
 $A_{{M'}}\pi_{{M,M'}}(X)\ngeq0$. 
 Since 
$\|X\|_{+}=
(\|X\|+n)/2$, 
from Lemma \ref{mark} we have 
\begin{equation}
 \|A_MX\|_-\geq p^{r-a}\|A_{{M'}}\pi_{{M,M'}}(X)\|_-
 -\frac{M^2}{2p^{b-a}\phi(M)}(\|X\|+n)\hspace{.75pt}.
\label{oldtimes1}\end{equation}
From (\ref{oldtimes1}), (\ref{oldtimes}) and from the induction hypothesis, 
it follows that
\begin{align}
\|A_MX\|_- &\geq  p^{r-a}\frac{g(n_1)}{G(M')} 
-\frac{M^2}{2p^{b-a}\varphi(M)}(n(F(M)+1)+c_n)\nonumber\\
 &\geq \frac{g(n-1)}{G(M)} 
-\frac{M^2}{2p^{b-a}\varphi(M)}(n(F(M)+1)+c_n),
\label{oldtimes2}\end{align}
where the last inequality holds since $g$ is a decreasing function and since 
$p^{r-a}G(M)\geq G(M')$.
Since (\ref{Itwouldbefinishedsoon}) implies that $b-a\geq r_n(M)$ (see (\ref{asyousow})),
from (\ref{stts0}) we get 
\begin{equation}g(n-1)-\frac{M^2G(M)}{2p^{b-a}\varphi(M)}(n(F(M)+1)+c_n)
>c_n\geq g(n),
\label{stts2}\end{equation}
where the last inequality follows from the definition of $g$ (see (\ref{saint1A}) and (\ref{saint3})).
From (\ref{oldtimes2}) and (\ref{stts2}), the claim follows. \qed
\section*{Acknowledgments}
I would like to thank Don Zagier for his encouragement in writing up this article.
I am also thankful to 
him 
as well as to
Sander Zwegers and Christian~Wei\ss \ 
for 
their comments on an earlier version of the manuscript. 
I would 
like to thank
Jeremy Rouse and John Webb for their 
very friendly 
communication regarding the result which we obtained independently. 
I~am grateful to the Max Planck Institute for Mathematics in Bonn and 
to the CIRM~:~FBK (International Center for Mathematical Research of the Bruno Kessler Foundation) in Trento
for providing me with office spaces and
supporting me with 
fellowships 
during the preparation of this article.

\bibliography{fixedweight-bibtex}
\bibliographystyle{IEEEtranS}
\nocite{*}

 \end{document}